\documentclass[ECP]{ejpecp} 

\usepackage{amssymb}
\usepackage{amsmath}
\usepackage{amsthm}
\usepackage{stmaryrd}
\usepackage{bbm}
\usepackage{enumitem}

\usepackage{fancyhdr}
\usepackage{hyperref}
\hypersetup{%
  pdftitle={BSDEs},%
  pdfsubject={BSDEs},%
  pdfauthor={ShengJun FAN, Ying Hu},%
  pdfkeywords={BSDEs},%
  pdfstartview=FitH,%
  CJKbookmarks=true,%
  bookmarksnumbered=true,%
  bookmarksopen=true,%
  colorlinks=true, linkcolor=blue, urlcolor=blue, citecolor=blue, %
}

\usepackage{cleveref}
\crefname{thm}{Theorem}{Theorems}
\crefname{pro}{Proposition}{Propositions}
\crefname{lem}{Lemma}{Lemmas}
\crefname{rmk}{Remark}{Remarks}
\crefname{cor}{Corollary}{Corollaries}
\crefname{dfn}{Definition}{Definitions}
\crefname{ex}{Example}{Examples}
\crefname{section}{Section}{Sections}
\crefname{subsection}{Subsection}{Subsections}

\SHORTTITLE{Scalar BSDEs with integrable terminal values: the critical case}

\TITLE{Existence and uniqueness of solution to scalar BSDEs with $L\exp\left(\mu\sqrt{2\log(1+L)}\right)$-integrable terminal values:\\ the critical case\thanks{Shengjun Fan is supported by the State Scholarship Fund from the China Scholarship Council (No. 201806425013). Ying Hu is partially supported by Lebesgue center of mathematics ``Investissements d'avenir" program-ANR-11-LABX-0020-01, by CAESARS-ANR-15-CE05-0024 and by MFG-ANR-16-CE40-0015-01.}} 

\AUTHORS{%
  Shengjun~Fan\footnote{School of Mathematics, China University of Mining and Technology, Xuzhou 221116, China. \BEMAIL{f_s_j@126.com}}
  \and 
  Ying~Hu\footnote{Univ. Rennes, CNRS, IRMAR-UMR6625, F-35000, Rennes, France. \EMAIL{ying.hu@univ-rennes1.fr}}}


\AMSSUBJ{60H10} 

\DOI{ }

\ABSTRACT{In \cite{HuTang2018ECP}, the existence of the solution is proved for a scalar linearly growing backward stochastic differential equation (BSDE) when the terminal value is $L\exp\left(\mu\sqrt{2\log(1+L)}\right)$-integrable for a positive parameter $\mu>\mu_0$ with a critical value $\mu_0$, and a counterexample is provided to show that the preceding integrability for $\mu<\mu_0$ is not sufficient to guarantee the existence of the solution.
Afterwards, the uniqueness result (with $\mu>\mu_0$) is also given in \cite{BuckdahnHuTang2018ECP} for the preceding BSDE under the uniformly Lipschitz condition of the generator. In this note, we prove that these two results still hold for the critical case: $\mu=\mu_0$.}

\newcommand{\To}{\rightarrow}
\newcommand{\as}{{\rm d}\mathbb{P}\times{\rm d} t-a.e.}

\newcommand{\ps}{\mathbb{P}-a.s.}

\newcommand{\F}{\mathcal{F}}
\newcommand{\E}{\mathbb{E}}

\newcommand{\s}{\mathcal{S}}
\newcommand{\lcal}{\mathcal{L}}
\newcommand{\mcal}{\mathcal{M}}

\newcommand{\T}{[0,T]}

\newcommand{\R}{{\mathbb R}}

\newcommand {\Dis}{\displaystyle}

\newtheorem{thm}{Theorem}[section]
\newtheorem{lem}[thm]{Lemma}
\newtheorem{pro}[thm]{Proposition}
\newtheorem{rmk}[thm]{Remark}

\begin{document}



\section{Introduction}
\label{sec:1-Introduction}
\setcounter{equation}{0}

Let us fix a positive integer $d$ and a positive real number $T>0$. Let $(B_t)_{t\in\T}$ be a $d$-dimensional standard Brownian motion defined on some complete probability space $(\Omega, \F, \mathbb{P})$, and $(\F_t)_{t\in\T}$ its natural filtration augmented by all $\mathbb{P}$-null sets of $\F$. For any two elements $x,y$ in $\R^d$, denote by $x\cdot y$ their scalar inner product. We recall that a real-valued and $(\F_t)$-progressively measurable process $(X_t)_{t\in\T}$ belongs to class (D) if the family of random variables $\{X_\tau: \tau\in \Sigma_T\}$ is uniformly integrable, where and hereafter $\Sigma_T$ denotes the set of all $(\F_t)$-stopping times $\tau$ valued in $\T$.

For any real number $p\geq 1$, let $L^p$ represent the set of (equivalent classes of) all real-valued and $\F_T$-measurable random variables $\xi$ such that $\E[|\xi|^p]<+\infty$, $\lcal^p$ the set of (equivalent classes of) all real-valued and $(\F_t)$-progressively measurable processes $(X_t)_{t\in\T}$ such that\vspace{-0.1cm}
$$
\|X\|_{\lcal^p}:=\left\{\E\left[\left(\int_0^T |X_t|{\rm d}t\right)^p\right]\right\}^{1/p}<+\infty,
$$
$\s^p$ the set of (equivalent classes of) all real-valued, $(\F_t)$-progressively measurable and continuous processes $(Y_t)_{t\in\T}$ such that
$$\|Y\|_{{\s}^p}:=\left(\E[\sup_{t\in\T} |Y_t|^p]\right)^{1/p}<+\infty,\vspace{0.2cm}$$
and $\mcal^p$ the set of (equivalent classes of) all $\R^d$-valued and $(\F_t)$-progressively measurable processes $(Z_t)_{t\in\T}$ such that
$$
\|Z\|_{\mcal^p}:=\left\{\E\left[\left(\int_0^T |Z_t|^2{\rm d}t\right)^{p/2}\right] \right\}^{1/p}<+\infty.\vspace{0.1cm}
$$

We study the following backward stochastic differential equation (BSDE for short):
\begin{equation}\label{eq:1}
  Y_t=\xi+\int_t^T g(s,Y_s,Z_s){\rm d}s-\int_t^T Z_s \cdot {\rm d}B_s, \ \ t\in\T,
\end{equation}
where $\xi$ is a real-valued and $\F_T$-measurable random variable called the terminal condition or terminal value, the function (here called the generator)
$g(\omega, t, y, z):\Omega\times\T\times\R\times\R^d \mapsto \R $
is $(\F_t)$-progressively measurable for each $(y,z)$ and continuous in $(y,z)$, and the pair of processes $(Y_t,Z_t)_{t\in\T}$ with values in $\R\times\R^d$ is called the solution of \eqref{eq:1}, which is $(\F_t)$-progressively measurable such that $\ps$, $t\mapsto Y_t$ is continuous, $t\mapsto Z_t$ belongs to $L^2(0,T)$, $t\mapsto g(t,Y_t,Z_t)$ is integrable, and verifies \eqref{eq:1}. By BSDE$(\xi,g)$, we mean the BSDE with terminal value $\xi$ and generator $g$.

The following two assumptions with respect to the generator $g$ will be used in this note. The first one is called the linear growth condition, and the second one is called the uniformly Lipschitz condition, which is obviously stronger than the linear growth condition.
\begin{enumerate}
\renewcommand{\theenumi}{(H\arabic{enumi})}
\renewcommand{\labelenumi}{\theenumi}
\item\label{H1} There exist two positive constants $\beta$ and $\gamma$ such that $\as$, for each $(y,z)\in \R\times\R^d$,
    $$
    |g(\omega,t,y,z)|\leq |g(\omega,t,0,0)|+\beta|y|+\gamma |z|;
    $$
\item\label{H2} There exist two positive constants $\beta$ and $\gamma$ such that $\as$, for all $(y^i,z^i)\in \R\times\R^d$, $i=1,2$,
    $$
    |g(\omega,t,y^1,z^1)-g(\omega,t,y^2,z^2)|\leq \beta|y^1-y^2|+\gamma |z^1-z^2|.
    $$
\end{enumerate}
Denote $g_0:=g(\cdot,0,0)$. It is well known that for $(\xi,g_0)\in L^p\times\lcal^p$ with some $p>1$, BSDE$(\xi,g)$ admits a minimal (maximal) solution $(Y_\cdot,Z_\cdot)$ in $\s^p\times\mcal^p$ if the generator $g$ satisfies assumption \ref{H1}, and the solution is unique in $\s^p\times\mcal^p$ if $g$ further satisfies assumption \ref{H2}. See e.g. \cite{PardouxPeng1990SCL,ElKarouiPengQuenez1997MF,LepeltierSanMartin1997SPL,
BriandDelyonHu2003SPA,FanJiang2012JAMC} for more details. However, for $(\xi,g_0)\in L^1\times\lcal^1$, one needs to restrict the generator $g$ to grow sub-linearly with respect to $z$, i.e., with some $q\in [0,1)$,
$$
|g(\omega,t,y,z)|\leq |g_0(\omega,t)|+\beta|y|+\gamma |z|^q,\ \ (\omega,t,y,z)\in \Omega\times\T\times\R\times\R^d,
$$
for BSDE$(\xi,g)$ to have a minimal (maximal) adapted solution and a unique solution  when the generator $g$ satisfies  \ref{H1} and \ref{H2} respectively. See for example \cite{BriandDelyonHu2003SPA,BriandHu2006PTRF,Fan2016SPA} for more details.

Recently, by applying the dual representation of solution to BSDE with convex generator, see for instance \cite{ElKarouiPengQuenez1997MF,Tang2006CRA,
DelbaenHuRichou2011AIHPPS}, to establish some a priori estimate and the localization procedure, the authors in \cite{HuTang2018ECP} proved the existence of a solution to BSDE$(\xi,g)$ when the generator $g$ satisfies \ref{H1} and the terminal value $(\xi,g_0)$ is $L\exp\left(\mu\sqrt{2\log(1+L)}\right)$-integrable for a positive parameter $\mu>\mu_0$ with a critical value $\mu_0=\gamma \sqrt{T}$, and showed by a counterexample that the conventionally expected $L\log L$ integrability and even the preceding integrability for a positive parameter $\mu<\mu_0$ is not enough for the existence of a solution to a BSDE with the generator $g$ satisfying \ref{H1}. Furthermore, by establishing some interesting properties of the function $\psi(x,\mu)=x\exp\left(\mu\sqrt{2\log(1+x)}\right)$ and observing the nice property of the obtained solution $Y$ that $\psi(|Y|,a)$ belongs to class (D) for some $a>0$, the authors in \cite{BuckdahnHuTang2018ECP} divided the whole interval $\T$ into a finite number of subintervals and proved the uniqueness of the solution to the preceding BSDE$(\xi,g)$ with the generator $g$ satisfying \ref{H2} and $\mu>\mu_0$.

In this note, we prove that the existence and uniqueness result obtained respectively in \cite{HuTang2018ECP} and \cite{BuckdahnHuTang2018ECP} is still true under the critical value case: $\mu=\gamma\sqrt{T}$, see \cref{thm:MainResult} in next section.

For the existence, in order to apply the localization procedure put forward initially in \cite{BriandHu2006PTRF}, the key is always to establish some uniform a priori estimate for the first process $Y^n_\cdot$ in the solution of the approximated BSDEs. For this, instead of applying the dual representation of solution to BSDE with convex generator, our whole idea consists in searching for an appropriate function $\phi(s, x; t)$ in order to apply It\^{o}-Tanaka's formula to $\phi(s,|Y^n_s|;t)$ on the time interval $s\in [t,\tau_n]$ with $(\F_t)$-stopping time $\tau_n$ valued in $[t,T]$. More specifically, we need to find a positive, continuous, strictly increasing and strictly convex function $\phi(s,x;t):[t,T]\times[0,+\infty)\mapsto (0,+\infty)$ with $t\in (0,T]$ satisfying
\begin{equation}\label{eq:2}
-\gamma \phi_x(s,x;t)|z| +{1\over 2}\phi_{xx}(s,x;t)|z|^2+\phi_s(s,x;t)\geq 0,\ \ (s,x,z)\in [t,T]\times [0,+\infty)\times\R^d,
\end{equation}
where and hereafter, for each $t\in (0,T]$, $\phi_s(\cdot,\cdot; t)$ denotes the first-order partial derivative of $\phi(\cdot,\cdot; t)$ with respect to the first variable, and $\phi_x(\cdot,\cdot; t)$ and $\phi_{xx}(\cdot,\cdot; t)$ respectively the first-order and second order partial derivative of $\phi(\cdot,\cdot; t)$ with respect to the second variable. Observe from the basic inequality $2ab\leq a^2+b^2$ that
$$
\begin{array}{lll}
\Dis -\gamma \phi_x(s,x;t)|z|+{1\over 2}\phi_{xx}(s,x;t)|z|^2&=& \Dis \phi_{xx}(s,x;t)\left( -{\gamma\phi_x(s,x;t)\over \phi_{xx}(s,x;t)}|z|+{1\over 2}|z|^2\right)\vspace{0.2cm}\\
&\geq & \Dis -{\gamma^2\over 2}{\phi^2_x(s,x;t)\over \phi_{xx}(s,x;t)}.
\end{array}
$$
Hence, it suffices if for each $t\in (0,T]$, the function $\phi(\cdot,\cdot;t)$ satisfies the following condition:\vspace{0.1cm}
\begin{equation}\label{eq:3}
-{\gamma^2\over 2}{\phi^2_x(s,x;t)\over \phi_{xx}(s,x;t)}+\phi_s(s,x;t)\geq 0,\ \ (s,x)\in [t,T]\times [0,+\infty).\vspace{0.1cm}
\end{equation}
Inspired by the investigation in \cite{HuTang2018ECP} and \cite{BuckdahnHuTang2018ECP}, we can choose the following function, for each $t\in (0,T]$,\vspace{-0.1cm}
\begin{equation}\label{eq:4}
\phi(s,x;t):=(x+e)\exp\left(\mu_s\sqrt{2\log(x+e)}+\int_t^s k_r{\rm d}r\right), \ (s,x)\in [t,T]\times [0,+\infty)
\end{equation}
to explicitly solve the inequality \eqref{eq:3}. We find that \eqref{eq:3} is satisfied for $\phi(s,x;t)$ when
\begin{equation}\label{eq:5}
\mu_s=\gamma\sqrt{s}\ \  {\rm and}\ \ k_r={\gamma\over 2}\left(\gamma+\sqrt{2\over r}\right).\vspace{-0.1cm}
\end{equation}

For the uniqueness of the solution to BSDE$(\xi,g)$, by virtue of two useful inequalities obtained in \cite{HuTang2018ECP}, we use a similar idea to that in \cite{BuckdahnHuTang2018ECP} to divide the whole
interval $\T$ into some sufficiently small subintervals and show successively the uniqueness of the solution in these subintervals. However, different from  \cite{BuckdahnHuTang2018ECP}, in our case the number of these subintervals, which are $[3T/4, T]$, $[3^2T/4^2, 3T/4]$, $[3^3T/4^3, 3^2T/4^2]$, $\cdots$, $[3^nT/4^n, 3^{n-1}T/4^{n-1}]$, $\cdots$, is infinite. Fortunately, observing that the left end points of these subintervals tend to 0 as $n\To \infty$ and in view of the continuity of the first process in the solution with respect to the time variable, we can obtain the uniqueness of the solution on the whole interval $\T$ by taking the limit.

\section{Existence and uniqueness}
\label{sec:2-Main results}

Define the function $\psi$:
\begin{equation}\label{eq:6}
\psi(x,\mu)=x\exp\left(\mu\sqrt{2\log(1+x)}\right),\ \ (x,\mu)\in [0,+\infty)\times [0,+\infty),
\end{equation}
which is introduced in \cite{HuTang2018ECP} and \cite{BuckdahnHuTang2018ECP}.\vspace{0.1cm}

The following existence and uniqueness theorem is the main result of this note.\vspace{0.1cm}

\begin{thm}\label{thm:MainResult}
Let $\xi$ be a terminal condition and $g$ be a generator which is continuous in $(y,z)$. If $g$ satisfies assumption \ref{H1} with parameters $\beta$ and $\gamma$, and\vspace{-0.1cm}
$$\psi\left(|\xi|+\int_0^T|g(t,0,0)|{\rm d}t,\ \gamma\sqrt{T}\right)\in L^1,\vspace{-0.1cm}$$
then BSDE$(\xi,g)$ admits a solution $(Y_t,Z_t)_{t\in\T}$ such that $\left(\psi\left(|Y_t|,\gamma\sqrt{t}\right)\right)_{t\in\T}$ belongs to class (D), and $\ps$, for each $t\in \T$,
\begin{equation}\label{eq:7}
|Y_t|\leq \psi(|Y_t|,\gamma\sqrt{t})\leq C\E\left[\left.\psi\left(|\xi|+\int_0^T|g(t,0,0)|{\rm d}t,\ \gamma\sqrt{T}\right)\right|\F_t\right]+C,
\end{equation}
where $C$ is a positive constant depending only on $(\beta,\gamma,T)$.

Furthermore, if $g$ satisfies assumption \ref{H2}, then BSDE$(\xi,g)$ admits a unique solution $(Y_t,Z_t)_{t\in\T}$ such that $\left(\psi\left(|Y_t|,\gamma\sqrt{t}\right)\right)_{t\in\T}$ belongs to class (D).\vspace{0.2cm}
\end{thm}

In order to prove the above theorem, we need the following lemmas and propositions. First, the following lemma have been proved, see Proposition 2.3 and the proof of Theorem 2.5 in \cite{BuckdahnHuTang2018ECP}.\vspace{0.1cm}

\begin{lem}\label{Lem:2.2}
We have the following assertions on $\psi$:
\begin{itemize}
\item [(i)] For each $x\geq 0$, $\psi(x,\cdot)$ is nondecreasing on $[0,+\infty)$.
\item [(ii)] For $\mu\geq 0$, $\psi(\cdot,\mu)$ is a positive, strictly increasing and strictly convex function on $[0,+\infty)$.
\item [(iii)] For $c\geq 1$, we have $\psi(cx,\mu)\leq \psi(c,\mu)\psi(x,\mu)$, for all $x,\mu\geq 0$.
\item [(iv)] For all $x_1,x_2,\mu\geq 0$, we have $\psi(x_1+x_2,\mu)\leq {1\over 2}\psi(2,\mu)\left[\psi(x_1,\mu)+\psi(x_2,\mu)\right].$\vspace{0.2cm}
\end{itemize}
\end{lem}

For each $t\in \T$, we define the following function $\varphi$, which will be applied by It\^{o}-Tanaka's formula later.
\begin{equation}\label{eq:8}
\varphi(s,x;t):=(x+e)\exp\left(\gamma\sqrt{2s\log(x+e)}+{\gamma\over 2}\int_t^s \left(\gamma+\sqrt{2\over r}\right){\rm d}r\right), \ (s,x)\in [t,T]\times [0,+\infty),
\end{equation}
which is the function $\phi$ in \eqref{eq:4} with $\mu_s$ and $k_r$ defined in \eqref{eq:5}. We have, for each $t\in (0,T]$ and each $(s,x)\in [t,T]\times [0,+\infty)$,
$$
\varphi_x(s,x;t)=\varphi(s,x;t)\ {\gamma\sqrt{s}+\sqrt{2\log(x+e)} \over (x+e)\sqrt{2\log(x+e)}}>0,
$$
$$
\varphi_{xx}(s,x;t)=\varphi(s,x;t)\ {\gamma\sqrt{s}\left(
2\log(x+e)+\gamma \sqrt{2s\log(x+e)}-1\right)\over (x+e)^2\left(\sqrt{2\log(x+e)}\right)^3}>0,
$$
and
$$
\varphi_s(s,x;t)={\gamma\varphi(s,x;t)\over 2}\  \left({\sqrt{2\log(x+e)}+\sqrt{2}\over \sqrt{s}}+\gamma\right)>0.\vspace{0.2cm}
$$

Moreover, we have the following proposition.\vspace{0.1cm}

\begin{pro}\label{Pro:2.3}
We have the following assertions on $\varphi$:
\begin{itemize}
\item [(i)] For $t\in\T$, $\varphi(\cdot,\cdot;t)$ is continuous on $[t,T]\times [0,+\infty)$; And, for all $t\in (0,T]$, $\varphi(\cdot,\cdot;t)\in \mathcal {C}^{1,2}([t,T]\times [0,+\infty))$;
\item [(ii)] For all $t\in (0,T]$, $\varphi(\cdot,\cdot;t)$ satisfies the inequality in \eqref{eq:2}, i.e.,
$$
-\gamma \varphi_x(s,x;t)|z| +{1\over 2}\varphi_{xx}(s,x;t)|z|^2+\varphi_s(s,x;t)\geq 0,\ \ (s,x,z)\in [t,T]\times [0,+\infty)\times\R^d.
$$
\end{itemize}
\end{pro}

\begin{proof}
The first assertion is obvious. In order to prove Assertion (ii), it suffices to prove that inequality \eqref{eq:3} holds for the function $\varphi(\cdot,\cdot;t)$ with $t\in (0,T]$ by virtue of the analysis in the introduction. In fact, by a simple computation, we have, for each $(s,x)\in [t,T]\times [0,+\infty)$,
$$
-{\gamma^2\over 2}{\varphi^2_x(s,x;t)\over \varphi_{xx}(s,x;t)}=-{\gamma\varphi(s,x;t)\over 2}{
\left(\gamma\sqrt{s}+\sqrt{2\log(x+e)}\right)^2 \sqrt{2\log(x+e)}
\over \sqrt{s}\left(2\log(x+e)+\gamma\sqrt{s}\sqrt{2\log(x+e)}-1\right)}.
$$
Define $v=\sqrt{2\log(x+e)}$. Then,
$$
-{\gamma^2\over 2}{\varphi^2_x(s,x;t)\over \varphi_{xx}(s,x;t)}+\varphi_s(s,x;t)=
{\gamma\varphi(s,x;t)\over 2}\left[{1\over \sqrt{s}}\left(v-{\left(\gamma\sqrt{s}+v\right)^2v\over v^2+\gamma\sqrt{s}v-1} \right)+\gamma+\sqrt{2\over s}\right].
$$
Furthermore, in view of the fact of $v\geq \sqrt{2}$, we know that
$$
\begin{array}{lll}
\Dis {\left(\gamma\sqrt{s}+v\right)^2v\over v^2+\gamma\sqrt{s}v-1}-v&=&\Dis {\gamma\sqrt{s}(v^2+\gamma\sqrt{s}v-1)+v+\gamma\sqrt{s}\over  v^2+\gamma\sqrt{s}v-1}\vspace{0.2cm}\\
&\leq& \Dis \gamma\sqrt{s}+{v+\gamma\sqrt{s}\over {1\over 2}(v^2+\gamma\sqrt{s}v)}= \gamma\sqrt{s}+{2\over v}\leq \gamma\sqrt{s}+\sqrt{2}.
\end{array}
$$
Hence, for each $t\in (0,T]$,\vspace{0.1cm}
$$
-{\gamma^2\over 2}{\varphi^2_x(s,x;t)\over \varphi_{xx}(s,x;t)}+\varphi_s(s,x;t)\geq 0,\ \ (s,x)\in [t,T]\times [0,+\infty).\vspace{0.1cm}
$$
Then, Assertion (ii) is proved, and the proof is complete.
\end{proof}

The two functions $\psi$ and $\varphi$ defined respectively on \eqref{eq:6} and \eqref{eq:8} has the following connection.\vspace{0.1cm}

\begin{pro}\label{pro:2.4}
There exists a universal constant $K>0$ depending only on $\gamma$ and $T$ such that for all $t\in\T$ and $(s,x)\in [t,T]\times[0,+\infty)$,
\begin{equation}\label{eq:9}
\psi(x,\gamma\sqrt{s})\leq \varphi(s,x;t)\leq K\psi(x,\gamma\sqrt{s})+K.
\end{equation}
In particular, by letting $s=t$, we have
\begin{equation}\label{eq:10}
\psi(x,\gamma\sqrt{t})\leq \varphi(t,x;t)\leq K\psi(x,\gamma\sqrt{t})+K,\ \ (t,x)\in [0,T]\times[0,+\infty).
\end{equation}
\end{pro}

\begin{proof}
The first inequality in \eqref{eq:9} is clear, and \eqref{eq:10} is a direct corollary of \eqref{eq:9}. We now prove the second inequality in \eqref{eq:9}. In fact, for each $t\in\T$ and $(s,x)\in [t,T]\times[1,+\infty)$,
$$
\begin{array}{lll}
\Dis{\varphi(s,x;t) \over\psi(x,\gamma\sqrt{s})+1}&=&\Dis {(x+e)\exp\left(\gamma\sqrt{2s\log(x+e)}+{\gamma\over 2}\int_t^s \left(\gamma+\sqrt{2\over r}\right){\rm d}r\right)\over x\exp\left(\gamma\sqrt{2s\log(1+x)}\right)+1}\vspace{0.2cm}\\
&\leq & {x+e\over x}\exp\left(\gamma\sqrt{T}\left(\sqrt{2\log(x+e)}-\sqrt{2\log(x+1)}\right)
+{\gamma^2T\over 2}+\gamma\sqrt{2T}\right)\vspace{0.2cm}\\
&=: & \Dis H_1(x,\gamma,T).
\end{array}
$$
And, in the case of $x\in [0,1]$,
$$
{\varphi(s,x;t) \over\psi(x,\gamma\sqrt{s})+1}\leq (e+1)\exp\left(\gamma\sqrt{2T\log(1+e)}+{\gamma^2T\over 2}+\gamma\sqrt{2T}\right)=: H_2(\gamma, T).
$$
Hence, for all $x\in [0,+\infty)$, we have
\begin{equation}\label{eq:11}
{\varphi(s,x;t) \over\psi(x,\gamma\sqrt{s})+1}\leq H_1(x,\gamma,T){\bf 1}_{x\geq 1}+H_2(\gamma, T){\bf 1}_{0\leq x<1}.
\end{equation}
With inequality \eqref{eq:11} in hand and in view of the fact that the function $H_1(x,\gamma, T)$ is continuous on $[1,+\infty)$ and tends to \vspace{-0.1cm}
$$\exp\left({\gamma^2T\over 2}+\gamma\sqrt{2T}\right)\vspace{-0.1cm}$$
as $x\To +\infty$, we obtain the second inequality in \eqref{eq:9}. The proof is complete.
\end{proof}

The following \cref{pro:2.5} establish some a priori estimate for the solution to a BSDE with an $L^p\ (p>1)$ terminal value and a linear-growth generator.\vspace{0.1cm}

\begin{pro}\label{pro:2.5}
Let $\xi$ be a terminal condition and $g$ be a generator which is continuous in $(y,z)$. If $g$ satisfies assumption \ref{H1} with parameters $\beta$ and $\gamma$, $\left(\xi, g(t,0,0)\right)\in L^p\times\lcal^p$ for some $p>1$, and $(Y_t,Z_t)_{t\in\T}$ is a solution in $\s^p\times \mcal^p$ to BSDE$(\xi,g)$, then $\ps$, for each $t\in \T$, we have
\begin{equation}\label{eq:12}
|Y_t|\leq \psi(|Y_t|,\gamma\sqrt{t})\leq C\E\left[\left.\psi\left(|\xi|+\int_0^T|g(t,0,0)|{\rm d}t,\ \gamma\sqrt{T}\right)\right|\F_t\right]+C,
\end{equation}
where $C$ is a positive constant depending only on $(\beta,\gamma,T)$, and $\psi$ is defined in \eqref{eq:6}.
\end{pro}

\begin{proof}
Note first that if $\left(\xi, g(t,0,0)\right)\in L^p\times\lcal^p$ for some $p>1$, then
$$\psi\left(|\xi|+\int_0^T|g(t,0,0)|{\rm d}t,\ \mu\right)\in L^1$$
for any $\mu\geq 0$, which has been shown in Remark 1.2 of \cite{HuTang2018ECP}. Define
$$
\bar Y_t:=e^{\beta t}|Y_t|+\int_0^t e^{\beta s}|g(s,0,0)|{\rm d}s\ \
{\rm and}\ \ \bar Z_t:=e^{\beta t}{\rm sgn}(Y_t)Z_t,\ \ t\in \T,
$$
where ${\rm sgn}(y)={\bf 1}_{y>0}-{\bf 1}_{y\leq 0}$. It then follows from It\^{o}-Tanaka's formula that, with $t\in\T$,
$$
\bar Y_t=\bar Y_T+\int_t^T e^{\beta s}\left({\rm sgn}(Y_s)g(s,Y_s,Z_s)-\beta |Y_s|-|g(s,0,0)|\right){\rm d}s-\int_t^T \bar Z_s \cdot {\rm d}B_s-\int_t^T e^{\beta s} {\rm d}L_s,
$$
where $L_\cdot$ denotes the local time of $Y_\cdot$ at $0$. Now, fix $t\in (0,T]$ and apply It\^{o}-Tanaka's formula to the process $\varphi(s, \bar Y_s; t)$, where the function $\varphi(\cdot,\cdot;t)$ is defined in \eqref{eq:8}, to derive, in view of assumption \ref{H1},
$$
\hspace*{-0.4cm}\begin{array}{lll}
&&\Dis {\rm d}\varphi(s,\bar Y_s;t)\vspace{0.1cm}\\
&=&\Dis e^{\beta s}\varphi_x(s,\bar Y_s;t)
\left(-{\rm sgn}(Y_s)g(s,Y_s,Z_s)+\beta |Y_s|+|g(s,0,0)|\right){\rm d}s+\varphi_x(s,\bar Y_s;t)\bar Z_s \cdot {\rm d}B_s\vspace{0.1cm}\\
&&\Dis +e^{\beta s}\varphi_x(s,\bar Y_s;t){\rm d}L_s+{1\over 2}e^{2\beta s}\varphi_{xx}(s,\bar Y_s;t)|Z_s|^2{\rm d}s+\varphi_s(s,\bar Y_s;t){\rm d}s
\vspace{0.1cm}\\
&\geq &\left[-\gamma e^{\beta s}\varphi_x(s,\bar Y_s;t)|Z_s|+{1\over 2}e^{2\beta s}\varphi_{xx}(s,\bar Y_s;t)|Z_s|^2+\varphi_s(s,\bar Y_s;t)\right]{\rm d}s+\varphi_x(s,\bar Y_s;t)\bar Z_s \cdot {\rm d}B_s.
\end{array}
$$
Furthermore, by letting $x=\bar Y_s$ and $z=e^{\beta s}Z_s$ in Assertion (ii) of \cref{Pro:2.3} we get that
\begin{equation}\label{eq:13}
{\rm d}\varphi(s,\bar Y_s;t)\geq \varphi_x(s,\bar Y_s;t)\bar Z_s \cdot {\rm d}B_s,\ \ s\in [t,T].\vspace{0.1cm}
\end{equation}
Let us consider, for each integer $n\geq 1$,  the following stopping time
$$
\tau_n:=\inf\left\{s\in [t,T]: \int_t^s \left[\varphi_x(r,\bar Y_r;t)\right]^2|\bar Z_r|^2{\rm d}r\geq n \right\}\wedge T,
$$
with the convention that $\inf \Phi=+\infty$. It follows from the inequality \eqref{eq:13} and the definition of $\tau_n$ that for each $t\in (0,T]$ and $n\geq 1$,
$$
\varphi(t,\bar Y_t;t)\leq \E\left[\left. \varphi(\tau_n,\bar Y_{\tau_n}; t) \right|\F_t\right].\vspace{-0.1cm}
$$
Thus, thanks to \cref{pro:2.4}, we know the existence of a positive constant $K$ depending only on $\gamma$ and $T$ such that
$$
\psi(\bar Y_t,\gamma\sqrt{t})\leq \varphi(t,\bar Y_t;t)\leq \E\left[\left. \varphi(\tau_n,\bar Y_{\tau_n};t) \right|\F_t\right]\leq K \E\left[\left. \psi(\bar Y_{\tau_n},\gamma\sqrt{\tau_n})\right|\F_t\right]+K,
$$
And, by virtue of \cref{Lem:2.2} and the fact that
$$
|Y_t|\leq |Y_t|+\int_0^t |g(s,0,0)|{\rm d}s\leq \bar Y_t\leq e^{\beta T}\left(|Y_t|+\int_0^t |g(s,0,0)|{\rm d}s\right),
$$
we obtain that for each $t\in (0,T]$ and $n\geq 1$,
$$
|Y_t|\leq\psi\left(|Y_t|,\gamma\sqrt{t}\right)\leq K \psi\left(e^{\beta T},\gamma\sqrt{{\tau_n}}\right)\E\left[\left. \psi\left(|Y_{\tau_n}|+\int_0^{\tau_n} |g(s,0,0)|{\rm d}s,\ \gamma\sqrt{{\tau_n}}\right)\right|\F_t\right]+K,
$$
from which the inequality \eqref{eq:12} follows for $t\in (0,T]$ by sending $n$ to infinity. Finally, in view of the continuity of $Y_\cdot$ and the martingale in the right side hand of \eqref{eq:12} with respect to the time variable $t$, we know that \eqref{eq:12} holds still true for $t=0$. The proposition is then proved.
\end{proof}

\begin{rmk}\label{rmk:2.6}
We specially point out that, to the best of our knowledge, under the critical case: $\mu=\gamma\sqrt{T}$, the method of the dual representation used in \cite{HuTang2018ECP} can not be applied to obtain the desired a priori estimate as that in \eqref{eq:12} at the time $t=0$.\vspace{0.2cm}
\end{rmk}

Now, we give the proof of the existence part of \cref{thm:MainResult}.

\begin{proof}[The proof of the existence part of \cref{thm:MainResult}]
Let us fix two positive integers $n$ and $p$. Set $\xi^{n,p}:=\xi^+\wedge n-\xi^-\wedge p$, $g^{n,p}(t,0,0):=g^+(t,0,0)\wedge n-g^-(t,0,0)\wedge p$ and $g^{n,p}(t,y,z):=g(t,y,z)-g(t,0,0)+g^{n,p}(t,0,0)$. As the terminal condition $\xi^{n,p}$ and $g^{n,p}(t,0,0)$ are bounded (hence square-integrable) and $g^{n,p}(t,y,z)$ is a continuous and linear-growth generator, in view of the existence result in \cite{LepeltierSanMartin1997SPL}, BSDE$(\xi^{n,p},g^{n,p})$ admits a minimal solution $(Y^{n,p}_\cdot, Z^{n,p}_\cdot)$ in $\s^2\times\mcal^2$. It then follows from \cref{pro:2.5} that there exists a positive constant $C$ depending only on $(\beta,\gamma,T)$ such that for each $t\in\T$ and each $n,p\geq 1$,
\begin{equation}\label{eq:14}
\hspace*{-0.2cm}\begin{array}{lll}
|Y^{n,p}_t|\leq \psi(|Y^{n,p}_t|,\gamma\sqrt{t})&\leq& \Dis C\E\left[\left.\psi\left(|\xi^{n,p}|+\int_0^T|g^{n,p}(t,0,0)|{\rm d}t,\ \gamma\sqrt{T}\right)\right|\F_t\right]+C\vspace{0.2cm}\\
&\leq& \Dis C\E\left[\left.\psi\left(|\xi|+\int_0^T|g(t,0,0)|{\rm d}t,\ \gamma\sqrt{T}\right)\right|\F_t\right]+C.
\end{array}
\end{equation}
Since $Y^{n,p}_\cdot$ is nondecreasing in $n$ and non-increasing in $n$ by the comparison theorem, then in view of \eqref{eq:14} and assumption \ref{H1}, by virtue of the localization method put forward in
\cite{BriandHu2006PTRF}, we know that there exists an $(\F_t)$-progressively measurable process $(Z_t)_{t\in\T} $ such that $(Y_\cdot:=\inf_p\sup_n Y^{n,p}_\cdot, \ Z_\cdot)$ is an adapted solution to BSDE$(\xi,g)$. Finally, sending $n$ and $p$ to infinity in \eqref{eq:14} yields the inequality \eqref{eq:7}, and then $\left(\psi\left(|Y_t|,\gamma\sqrt{t}\right)\right)_{t\in\T}$ belongs to class (D). The proof is complete.
\end{proof}

\begin{rmk}\label{rmk:2.7}
From the above proof, it is easy to see that the linear-growth assumption \ref{H1} in \cref{thm:MainResult} and \cref{pro:2.5} can be easily weakened to the following one-sided linear-growth assumption: There exist two real constants $\beta\geq 0$, $\gamma>0$ and a nonnegative, real-valued and $(\F_t)$-progressively measurable process $(f_t)_{t\in \T}$ such that $\as$, for each $(y,z)\in \R\times\R^d$,
$$
{\rm sgn}(y)g(\omega,t,y,z)\leq f_t(\omega)+\beta|y|+\gamma |z|\ \ {\rm and}\ \ |g(\omega,t,y,z)|\leq f_t(\omega)+ h(|y|)+\gamma |z|,
$$
where $h(\cdot)$ is a deterministic
continuous nondecreasing function with $h(0)=0$. In this case, $|g(t,0,0)|$ in the conditions of \cref{thm:MainResult} and \cref{pro:2.5} only needs to be replaced with $f_t$.\vspace{0.2cm}
\end{rmk}

In order to prove the uniqueness part of \cref{thm:MainResult}, we need the following two lemmas, which can be found in \cite{HuTang2018ECP}. \vspace{0.1cm}

\begin{lem}\label{lem:2.8}
For each $x\in \R$, $y\geq 0$ and $\mu>0$, we have
$$
e^xy\leq e^{x^2\over 2\mu^2}+e^{2\mu^2}\psi(y,\mu),
$$
where the function $\psi$ is defined in \eqref{eq:6} again.\vspace{0.2cm}
\end{lem}

\begin{lem}\label{lem:2.9}
Let $(q_t)_{t\in\T}$ be a $d$-dimensional and $(\F_t)$-progressively measurable process with $|q_\cdot|\leq \gamma$ almost surely. For each $t\in\T$, if $0\leq \lambda<{1 \over 2\gamma^2 (T-t)}$, then
$$
\E\left[\left.e^{\lambda \left|\int_t^Tq_s\cdot {\rm d}B_s\right|^2}\right|\F_t\right]\leq {1\over \sqrt{1-2\lambda\gamma^2(T-t)}}.\vspace{0.2cm}
$$
\end{lem}

Now, we give the proof of the uniqueness part of \cref{thm:MainResult}.

\begin{proof}[The proof of the uniqueness part of \cref{thm:MainResult}]
Let $g$ satisfy assumption \ref{H2}, and for $i=1,2$, let $(Y^i_t,Z^i_t)_{t\in\T}$ be a solution of BSDE$(\xi,g)$ such that $\left(\psi\left(|Y^i_t|,\gamma\sqrt{t}\right)\right)_{t\in\T}$ belongs to class (D). Define $\delta Y_\cdot:=Y^1_\cdot-Y^2_\cdot$ and $\delta Z_\cdot:=Z^1_\cdot-Z^2_\cdot$. Then the pair $(\delta Y_\cdot,\delta Z_\cdot)$ verifies the following BSDE:
$$
  \delta Y_t=\int_t^T \left(u_s\delta Y_s+v_s\cdot\delta Z_s\right) {\rm d}s-\int_t^T \delta Z_s \cdot {\rm d}B_s, \ \ t\in\T,
$$
where $g(s,Y^1_s,Z^1_s)-g(s,Y^2_s,Z^2_s)=u_s\delta Y_s+v_s\cdot\delta Z_s$ with a pair of $(\F_t)$-progressively measurable process $(u_\cdot,v_\cdot)$ such that $|u_s|\leq \beta$ and $|v_s|\leq \gamma$ by a standard linearization procedure. For each $t\in (0,T]$ and each positive integer $n\geq 1$, define the following stopping times:\vspace{-0.1cm}
$$
\sigma_n:=\inf\left\{s\in [t,T]:\ |\delta Y_s|+\int_t^s |\delta Z_r|^2 {\rm d}r\geq n\right\}\wedge T,
$$
with the convention that $\inf \Phi=+\infty$. Then,\vspace{0.2cm}
$$
\delta Y_t=\E\left[\left.e^{\int_t^{\sigma_n}
u_s{\rm d}s+\int_t^{\sigma_n}v_s\cdot {\rm d}B_s-{1\over 2}\int_t^{\sigma_n}|v_s|^2{\rm d}s }\delta Y_{\sigma_n}  \right|\F_t\right].
$$
Therefore,
\begin{equation}\label{eq:15}
\left|\delta Y_t\right|\leq e^{\beta T}\E\left[\left.e^{\int_t^{\sigma_n}v_s\cdot {\rm d}B_s}\left|\delta Y_{\sigma_n}\right|\right|\F_t\right].\vspace{0.2cm}
\end{equation}
Furthermore, by virtue of \cref{lem:2.8} we know that for each $n\geq 1$,
\begin{equation}\label{eq:16}
e^{\int_t^{\sigma_n}v_s\cdot {\rm d}B_s}\left|\delta Y_{\sigma_n}\right|\leq e^{{1\over 2\gamma^2t}\left(\int_t^{\sigma_n}v_s\cdot {\rm d}B_s\right)^2}+e^{2\gamma^2t}\psi(\left|\delta Y_{\sigma_n}\right|,\gamma\sqrt{t}), \ \ t\in (0,T].
\end{equation}
And, it follows from \cref{lem:2.9} that for all $n\geq 1$,
$$
\E\left[\left| e^{{1\over 2\gamma^2t}\left(\int_t^{\sigma_n}v_s\cdot {\rm d}B_s\right)^2}  \right|^2\right]=\E\left[e^{{1\over \gamma^2t}\left(\int_t^{\sigma_n}v_s\cdot {\rm d}B_s\right)^2}  \right]\leq {1\over \sqrt{1-{2(T-t)\over t}}}\leq \sqrt{3},\ \ t\in [3T/4,T],
$$
and, thus, the family of random variables $e^{{1\over 2\gamma^2t}\left(\int_t^{\sigma_n}v_s\cdot {\rm d}B_s\right)^2}$ is uniformly integrable on the time interval $[3T/4,T]$. On the other hand, in view of \cref{Lem:2.2}, observe that for all $n\geq 1$,
$$
\begin{array}{lll}
\Dis e^{2\gamma^2t}\psi(\left|\delta Y_{\sigma_n}\right|,\gamma\sqrt{t})&\leq & \Dis  e^{2\gamma^2T}\psi(\left|Y^1_{\sigma_n}\right|+\left|Y^2_{\sigma_n}\right|,
\gamma\sqrt{\sigma_n})\vspace{0.2cm}\\
&\leq & \Dis {e^{2\gamma^2T}\psi(2,\gamma\sqrt{T}) \over 2}\left[\psi(\left| Y^1_{\sigma_n}\right|,\gamma\sqrt{\sigma_n})+\psi(\left| Y^2_{\sigma_n}\right|,\gamma\sqrt{\sigma_n})\right],\ \ t\in\T.
\end{array}
$$
Thus, from \eqref{eq:16} we can conclude that, for $t\in [3T/4,T]$, the family of random variables $e^{\int_t^{\sigma_n}v_s\cdot {\rm d}B_s}\left|\delta Y_{\sigma_n}\right|$ is uniformly integrable. Consequently, by letting $n\To\infty$ in the inequality \eqref{eq:15} we have $\delta Y_\cdot=0$ on the interval $[3T/4,T]$. It is clear that $\delta Z_\cdot=0$ on the interval $[3T/4,T]$. The uniqueness of the solution on the interval $[3T/4,T]$ is obtained. In a same way, we successively have the uniqueness on the intervals $[3^2T/4^2, 3T/4]$, $[3^3T/4^3, 3^2T/4^2]$, $\cdots$, $[3^pT/4^p, 3^{p-1}T/4^{p-1}]$, $\cdots$. Finally, in view of the continuity of process $\delta Y_t$ with respect to the time variable $t$, we obtain the uniqueness on the whole interval $\T$ by sending $p$ to infinity. The proof is then complete.\vspace{0.1cm}
\end{proof}

\begin{rmk}\label{rmk:2.10}
By a similar analysis to Remark 2.6 in \cite{BuckdahnHuTang2018ECP}, we know that the uniformly Lipschitz assumption \ref{H2} in \cref{thm:MainResult} can be relaxed to the following monotone assumption:
$$
{\rm sgn}(y_1-y_2)\left(g(\omega,t,y^1,z)-g(\omega,t,y^2,z)\right)\leq \beta |y^1-y^2|
$$
and
$$
|g(\omega,t,y,z^1)-g(\omega,t,y,z^2)|\leq \gamma |z^1-z^2|.
$$
\end{rmk}

\vspace{0.1cm}




\begin{thebibliography}{99}

\bibitem{BriandDelyonHu2003SPA}
Philippe Briand, Bernard Delyon, Ying Hu, Etienne Pardoux, and L.~Stoica,
  \emph{{$L^p$} solutions of backward stochastic differential equations},
  Stochastic Process. Appl. \textbf{108} (2003), no.~1, 109--129. \MR{2008603}

\bibitem{BriandHu2006PTRF}
Philippe Briand and Ying Hu, \emph{{BSDE} with quadratic growth and unbounded
  terminal value}, Probab. Theory Related Fields \textbf{136} (2006), no.~4,
  604--618. \MR{2257138}

\bibitem{BuckdahnHuTang2018ECP}
Rainer Buckdahn, Ying Hu, and Shanjian Tang, \emph{Existence of solution to
  scalar {BSDE}s with ${L}\exp\left(\mu\sqrt{2\log(1+L)}\right)$-integrable
  terminal values}, Electron. Commun. Probab. \textbf{23} (2018), Paper No. 59,
  8pp. \MR{3863915}

\bibitem{DelbaenHuRichou2011AIHPPS}
Freddy Delbaen, Ying Hu, and Adrien Richou, \emph{On the uniqueness of
  solutions to quadratic {BSDE}s with convex generators and unbounded terminal
  conditions}, Ann. Inst. Henri Poincar\'{e} Probab. Stat. \textbf{47} (2011),
  559--574. \MR{2814423}

\bibitem{ElKarouiPengQuenez1997MF}
Nicole El~Karoui, Shige Peng, and Marie~Claire Quenez, \emph{Backward
  stochastic differential equations in finance}, Math. Finance \textbf{7}
  (1997), no.~1, 1--71. \MR{1434407}

\bibitem{Fan2016SPA}
Shengjun Fan, \emph{Bounded solutions, ${L}^{p}\ (p>1)$ solutions and ${L}^1$
  solutions for one-dimensional {BSDE}s under general assumptions}, Stochastic
  Process. Appl. \textbf{126} (2016), 1511--1552. \MR{3473104}

\bibitem{FanJiang2012JAMC}
Shengjun Fan and Long Jiang, \emph{${L}^p$ $(p>1)$ solutions for
  one-dimensional {BSDE}s with linear-growth generators}, Journal of Applied
  Mathematics and Computing \textbf{38} (2012), no.~1--2, 295--304. \MR{2886682}

\bibitem{HuTang2018ECP}
Ying Hu and Shanjian Tang, \emph{Existence of solution to scalar {BSDE}s with
  ${L}\exp\sqrt{{2\over \lambda}\log(1+L)}$-integrable terminal values},
  Electron. Commun. Probab. \textbf{23} (2018), Paper No. 27, 11pp. \MR{3798238}

\bibitem{LepeltierSanMartin1997SPL}
Jean-Pierre Lepeltier and Jaime San~Martin, \emph{Backward stochastic
  differential equations with continuous coefficient}, Statist. Probab. Lett.
  \textbf{32} (1997), no.~4, 425--430. \MR{1602231}

\bibitem{PardouxPeng1990SCL}
Etienne Pardoux and Shige Peng, \emph{Adapted solution of a backward stochastic
  differential equation}, Syst. Control Lett. \textbf{14} (1990), no.~1,
  55--61. \MR{1037747}

\bibitem{Tang2006CRA}
Shanjian Tang, \emph{Dual representation as stochastic differential games of
  backward stochastic differential equations and dynamic evaluations}, C. R.
  Math. Acad. Sci. Paris \textbf{342} (2006), 773--778. \MR{2227758}


\end{thebibliography}





\end{document}